\documentclass[11pt]{article} 
\usepackage{amsmath,amssymb,amsthm}  
\usepackage{graphicx}  

\allowdisplaybreaks

\begin{document}

\def\fl#1{\left\lfloor#1\right\rfloor}
\def\cl#1{\left\lceil#1\right\rceil}
\def\ang#1{\left\langle#1\right\rangle}
\def\stf#1#2{\left[#1\atop#2\right]} 
\def\sts#1#2{\left\{#1\atop#2\right\}}
\def\eul#1#2{\left\langle#1\atop#2\right\rangle}
\def\N{\mathbb N}
\def\Z{\mathbb Z}
\def\R{\mathbb R}
\def\C{\mathbb C}

\newtheorem{theorem}{Theorem}
\newtheorem{Prop}{Proposition}
\newtheorem{Cor}{Corollary}
\newtheorem{Lem}{Lemma} 

\newenvironment{Rem}{\begin{trivlist} \item[\hskip \labelsep{\it
Remark.}]\setlength{\parindent}{0pt}}{\end{trivlist}}

\title{Sylvester power and weighted sums on the Frobenius set in arithmetic progression
}

\author{
Takao Komatsu 
\\
\small Department of Mathematical Sciences, School of Science\\[-0.8ex]
\small Zhejiang Sci-Tech University\\[-0.8ex]
\small Hangzhou 310018 China\\[-0.8ex]
\small \texttt{komatsu@zstu.edu.cn}
}

\date{
\small MR Subject Classifications: Primary 11D07; Secondary 05A15, 05A17, 05A19, 11B68, 11D04, 11P81 
}

\maketitle
 
\begin{abstract} 
Let $a_1,a_2,\dots,a_k$ be positive integers with $\gcd(a_1,a_2,\dots,a_k)=1$. Frobenius number is the largest positive integer that is NOT representable in terms of $a_1,a_2,\dots,a_k$. When $k\ge 3$, there is no explicit formula in general, but some formulae may exist for special sequences $a_1,a_2,\dots,a_k$, including, those forming arithmetic progressions and their modifications.    
In this paper, we give formulae for the power and weighted sum of nonrepresentable positive integers. As applications, we show explicit expressions of these sums for $a_1,a_2,\dots,a_k$ forming arithmetic progressions. 
\end{abstract}

\section{Introduction}  

The {\it Coin Exchange Problem} (or Postage Stamp Problem / Chicken McNugget Problem)  has a long history and is one of the problems that has attracted many people as well as experts. 
Given positive integers $a_1,\dots,a_k$ with $\gcd(a_1,\dots,a_k)=1$, it is well-known that all sufficiently large $n$ can be represented as a nonnegative integer combination of $a_1,\dots,a_k$. Nowadays, it is most known as 
the {\it Frobenius Problem}, which is to determine the largest positive integer that is NOT representable as a nonnegative integer combination of given positive integers that are coprime (see \cite{ra05} for general references). This number is denoted by $g(a_1,\dots,a_k)$ and often called {\it Frobenius number}. 

Let $n(a_1,\dots,a_k)$ be the number of positive integers with no nonnegative integer representation by $a_1,\dots,a_k$. It is sometimes called {\it Sylvester number}.  

According to Sylvester, for positive integers $a$ and $b$ with $\gcd(a,b)=1$,  
\begin{align*}
g(a,b)&=(a-1)(b-1)-1\quad{\rm \cite{sy1884}}\,,\\
n(a,b)&=\frac{1}{2}(a-1)(b-1)\quad{\rm \cite{sy1882}}\,. 
\end{align*}

There are many kinds of problems related to the Frobenius problem. The problems for the number of solutions (e.g., \cite{tr00}), and the sum of integer powers of values the gaps in numerical semigroups (e.g., \cite{bs93,fr07,fks}) are popular. 
One of other famous problems is about the so-called {\it Sylvester sums} 
$$
s(a_1,\dots,a_k):=\sum_{n\in{\rm NR}(a_1,\dots,a_k)}n 
$$ 
(see, e.g., \cite[\S 5.5]{ra05}, \cite{tu06} and references therein), where ${\rm NR}(a_1,\dots,a_k)$ denotes the set of positive integers without nonnegative integer representation by $a_1,\dots,a_k$. In addition, denote the set of positive integers with nonnegative integer representation by $a_1,\dots,a_k$ by ${\rm R}(a_1,\dots,a_k)$. It is harder to obtain the Sylvester number than the Frobenius number, and even harder to obtain the Sylvester sum. Finally, long time after Sylvester, 
Brown and Shiue \cite{bs93} found the exact value for positive integers $a$ and $b$ with $\gcd(a,b)=1$,  
\begin{equation}
s(a,b)=\frac{1}{12}(a-1)(b-1)(2 a b-a-b-1)\,. 
\label{brown}
\end{equation} 
R\o dseth \cite{ro94} generalized Brown and Shiue's result by giving a closed form for $
s_\mu(a,b):=\sum_{n\in{\rm NR}(a,b)}n^\mu$,  
where $\mu$ is a positive integer.  

When $k=2$, there exist beautiful closed forms for Frobenius numbers, Sylvester numbers and Sylvester sums, but 
when $k\ge 3$, exact determination of these numbers is extremely difficult.  
The Frobenius number cannot be given by closed formulas of a certain type (Curtis (1990) \cite{cu90}), the problem to determine $F(a_1,\dots,a_k)$ is NP-hard under Turing reduction (see, e.g., Ram\'irez Alfons\'in \cite{ra05}). 
Nevertheless, one convenient formula is found by Johnson \cite{jo60}. One analytic approach to the Frobenius number can be seen in \cite{bgk01,ko03}.

Though closed forms for general case are hopeless for $k\ge 3$, several formulae for Frobenius numbers, Sylvester numbers and Sylvester sums have been considered under special cases.  For example, one of the best expositions for the Frobenius number in three variables can be seen in \cite{tr17}. For general $k\ge 3$, the Frobenius number and the Sylvester number for some special cases are calculated, including arithmetic sequences and geometric-like sequences (e.g., \cite{br42,op08,ro56,se77}). 

In fact, by introducing the other numbers, it is possible to determine the functions $g(A)$, $n(A)$ and $s(A)$ for the set of positive integers $A:=\{a_1,a_2,\dots,a_k\}$ with $\gcd(a_1,a_2,\dots,a_k)=1$. 

For each integer $i$ with $1\le i\le a_1-1$, there exists a least positive integer $m_i\equiv i\pmod{a_1}$ with $m_i\in{\rm R}(a_1,a_2,\dots,a_k)$. For convenience, we set $m_0=0$.  With the aid of such a congruence consideration modulo $a_1$, very useful results are established.  

\begin{Lem}  
We have 
\begin{align*}
g(a_1,a_2,\dots,a_k)&=\left(\max_{1\le i\le a_1-1}m_i\right)-a_1\,,\quad{\rm \cite{bs62}}\\ 
n(a_1,a_2,\dots,a_k)&=\frac{1}{a_1}\sum_{i=1}^{a_1-1}m_i-\frac{a_1-1}{2}\,,\quad{\rm \cite{se77}}\\ 
s(a_1,a_2,\dots,a_k)&=\frac{1}{2 a_1}\sum_{i=1}^{a_1-1}m_i^2-\frac{1}{2}\sum_{i=1}^{a_1-1}m_i+\frac{a_1^2-1}{12}\,.\quad{\rm \cite{tr08}}
\end{align*}
\label{lem1} 
\end{Lem} 
Note that the third formula appeared with a typo in \cite{tr08}, and it has been corrected in \cite{pu18,tr17b}. 
\bigskip

In this paper, we give formulae for the power and weighted sum of nonrepresentable positive integers. As applications, we show explicit expressions of these sums for $a_1,a_2,\dots,a_k$ forming arithmetic progressions. Theorem \ref{lem2} is a general formula for power sum of nonrepresentable positive integers in general sequences $a_1,a_2,\dots,a_k$, that is 
$$
s_\mu(a_1,\dots,a_k):=\sum_{n\in{\rm NR}(a_1,\dots,a_k)}n^\mu  
$$ 
for a positive integer $\mu$.  
This may be called {\it Sylvester power sum}. 
This includes the third formula in Lemma \ref{lem1} as a special case. Theorem \ref{lem2} is useful to obtain an expression of each concrete sequence. Nevertheless, any elegant formula may be yielded from the sequence with a good pattern.  Indeed, Theorem \ref{th:pw-arith} is an explicit formula for power sum of nonrepresentable positive integers when the sequence $a_1,a_2,\dots,a_k$ forms an arithmetic progression. 

Next, we study the so-called {\it Sylvester weighted power sum}, that is 
$$
s_\mu^{(\lambda)}(a_1,\dots,a_k):=\sum_{n\in{\rm NR}(a_1,\dots,a_k)}\lambda^n n^\mu  
$$ 
for a positive integer $\mu$ with weight $\lambda(\ne 0,1)$. Weighted sums include the so-called alternate sums in special cases.  
Lemma \ref{lem-hh} is a general formula for weighted power sum of nonrepresentable positive integers in general sequences $a_1,a_2,\dots,a_k$ with $\lambda^{a_1}\ne 1$. Namely, $\lambda$ is not the $a_1$-th root of unity. As application, in Theorem \ref{th:wpw-arith}, we can get an explicit expression in the case of arithmetic progressions. 
 
When $\lambda^{a_1}=1$, Theorem \ref{lem-hh-a1} provides the counterpart formulae of Lemma \ref{lem-hh}. Theorem \ref{th:wpw-arith-d1} and Theorem \ref{th:wpw-arith-a1} show explicit formulae in exceptional cases, which are not included in Theorem \ref{th:wpw-arith}.   
Several examples illustrate and support our formulae.

\section{Power sums}  

We need the formula for the power sum of nonrepresentable numbers. Here, Bernoulli numbers $B_n$ are defined by the generating function 
$$
\frac{x}{e^x-1}=\sum_{n=0}^\infty B_n\frac{x^n}{n!}\,. 
$$  
The first few values of Bernoulli numbers are given by 
$$
B_0=1,~B_1=-\frac{1}{2},~B_2=\frac{1}{6},~B_4=-\frac{1}{30},~B_6=\frac{1}{42},~B_8=-\frac{1}{30},~B_{10}=\frac{5}{66} 
$$ 
with $B_n=0$ for odd $n\ge 3$. 

\begin{theorem}  
For $\mu\ge 0$, we have 
\begin{align*} 
&s_\mu(a_1,a_2,\dots,a_k):=\sum_{n\in{\rm NR}(a_1,a_2,\dots,a_k)}n^\mu\\ 
&=\frac{1}{\mu+1}\sum_{\kappa=0}^{\mu}\binom{\mu+1}{\kappa}B_{\kappa}a_1^{\kappa-1}\sum_{i=1}^{a_1-1}m_i^{\mu+1-\kappa} 
+\frac{B_{\mu+1}}{\mu+1}(a_1^{\mu+1}-1)\,. 
\end{align*}
\label{lem2} 
\end{theorem}

For $\mu=0$, we have 
$$
n(a_1,a_2,\dots,a_k)=\frac{1}{a_1}\sum_{i=1}^{a_1-1}m_i-\frac{a_1-1}{2}\,.\quad{\rm \cite{se77}}
$$ 
For $\mu=1,2,\dots,8$, we have 
\begin{align*}  
s_1(a_1,\dots,a_k)&=\frac{1}{2 a_1}\sum_{i=1}^{a_1-1}m_i^2-\frac{1}{2}\sum_{i=1}^{a_1-1}m_i+\frac{a_1^2-1}{12}\,,\quad({\rm \cite{tr08}})\\ 
s_2(a_1,\dots,a_k)&=\frac{1}{3 a_1}\sum_{i=1}^{a_1-1}m_i^3-\frac{1}{2}\sum_{i=1}^{a_1-1}m_i^2+\frac{a_1}{6}\sum_{i=1}^{a_1-1}m_i\,,\\ 
s_3(a_1,\dots,a_k)&=\frac{1}{4 a_1}\sum_{i=1}^{a_1-1}m_i^4-\frac{1}{2}\sum_{i=1}^{a_1-1}m_i^3+\frac{a_1}{4}\sum_{i=1}^{a_1-1}m_i^2-\frac{a_1^4-1}{120}\,,\\ 
s_4(a_1,\dots,a_k)&=\frac{1}{5 a_1}\sum_{i=1}^{a_1-1}m_i^5-\frac{1}{2}\sum_{i=1}^{a_1-1}m_i^4+\frac{a_1}{3}\sum_{i=1}^{a_1-1}m_i^3-\frac{a_1^3}{30}\sum_{i=1}^{a_1-1}m_i\,,\\ 
s_5(a_1,\dots,a_k)&=\frac{1}{6 a_1}\sum_{i=1}^{a_1-1}m_i^6-\frac{1}{2}\sum_{i=1}^{a_1-1}m_i^5+\frac{5 a_1}{12}\sum_{i=1}^{a_1-1}m_i^4-\frac{a_1^3}{12}\sum_{i=1}^{a_1-1}m_i^2\\
&\quad +\frac{a_1^6-1}{252}\,,\\ 
s_6(a_1,\dots,a_k)&=\frac{1}{7 a_1}\sum_{i=1}^{a_1-1}m_i^7-\frac{1}{2}\sum_{i=1}^{a_1-1}m_i^6+\frac{a_1}{2}\sum_{i=1}^{a_1-1}m_i^5-\frac{a_1^3}{6}\sum_{i=1}^{a_1-1}m_i^3\\
&\quad +\frac{a_1^5}{42}\sum_{i=1}^{a_1-1}m_i\,,\\ 
s_7(a_1,\dots,a_k)&=\frac{1}{8 a_1}\sum_{i=1}^{a_1-1}m_i^8-\frac{1}{2}\sum_{i=1}^{a_1-1}m_i^7+\frac{7 a_1}{12}\sum_{i=1}^{a_1-1}m_i^6-\frac{7 a_1^3}{24}\sum_{i=1}^{a_1-1}m_i^4\\
&\quad +\frac{a_1^5}{12}\sum_{i=1}^{a_1-1}m_i^2-\frac{a_1^8-1}{240}\,,\\ 
s_8(a_1,\dots,a_k)&=\frac{1}{9 a_1}\sum_{i=1}^{a_1-1}m_i^9-\frac{1}{2}\sum_{i=1}^{a_1-1}m_i^8+\frac{2 a_1}{3}\sum_{i=1}^{a_1-1}m_i^7-\frac{7 a_1^3}{15}\sum_{i=1}^{a_1-1}m_i^5\\
&\quad +\frac{2 a_1^5}{9}\sum_{i=1}^{a_1-1}m_i^3-\frac{a_1^7}{30}\sum_{i=1}^{a_1-1}m_i\,. 
\end{align*}

\begin{proof}[Proof of Theorem \ref{lem2}] 
For convenience, put $a=a_1$. 
Since 
\begin{align*}
\sum_{j=1}^\ell j^n&=\sum_{\kappa=0}^n\binom{n}{\kappa}(-1)^\kappa B_\kappa\frac{\ell^{n+1-\kappa}}{n+1-\kappa}\\
&=\frac{1}{n+1}\sum_{\kappa=0}^n\binom{n+1}{\kappa}(-1)^\kappa B_\kappa\ell^{n+1-\kappa}\\
&=\frac{1}{n+1}\sum_{\kappa=0}^n\binom{n+1}{\kappa}B_\kappa(\ell+1)^{n+1-\kappa}\\
&=\sum_{\kappa=0}^n\binom{n}{\kappa}B_{n-\kappa}\frac{(\ell+1)^{\kappa+1}}{\kappa+1}
\end{align*}
for $\ell_i=(m_i-i)/a$, we have 
\begin{align*}
&s_\mu(a_1,a_2,\dots,a_k)\\
&=\sum_{i=1}^{a-1}\sum_{j=1}^{\ell_i}(m_i-j a)^\mu\\ 
&=\sum_{i=1}^{a-1}\sum_{j=1}^{\ell_i}\sum_{n=0}^\mu\binom{\mu}{n}m_i^{\mu-n}j^n(-a)^n\\ 
&=\sum_{i=1}^{a-1}\sum_{n=0}^\mu\binom{\mu}{n}m_i^{\mu-n}(-a)^n\sum_{\kappa=0}^n\binom{n}{\kappa}(-1)^\kappa B_\kappa\frac{1}{n+1-\kappa}\left(\frac{m_i-i}{a}\right)^{n+1-\kappa}\\
&=\sum_{n=0}^\mu\binom{\mu}{n}\sum_{\kappa=0}^n\binom{n}{\kappa}\frac{(-1)^{n-\kappa}a^{\kappa-1}B_\kappa}{n+1-\kappa}\sum_{i=1}^{a-1}\sum_{j=0}^{n+1-\kappa}\binom{n+1-\kappa}{j}m_i^{\mu+1-\kappa-j}(-i)^j\,. 
\end{align*}
Consider the terms for $j=0$. Since 
\begin{align*}  
\sum_{n=\kappa}^\mu\binom{\mu-\kappa+1}{n+1-\kappa}(-1)^{n-\kappa}&=1+\sum_{j=0}^{\mu-\kappa+1}\binom{\mu-\kappa+1}{j}(-1)^{j-1}=1\,, 
\end{align*}
we get 
$$
\sum_{n=\kappa}^\mu\binom{\mu}{n}\binom{n}{\kappa}\frac{(-1)^{n-\kappa}}{n+1-\kappa}=\frac{1}{\mu+1}\binom{\mu+1}{\kappa}\,. 
$$ 
Hence, 
\begin{align*}  
&\sum_{n=0}^\mu\binom{\mu}{n}\sum_{\kappa=0}^n\binom{n}{\kappa}\frac{(-1)^{n-\kappa}a^{\kappa-1}B_\kappa}{n+1-\kappa}\sum_{i=1}^{a-1}m_i^{\mu+1-\kappa}\\ 
&=\sum_{\kappa=0}^\mu\sum_{n=\kappa}^\mu\binom{\mu}{n}\binom{n}{\kappa}\frac{(-1)^{n-\kappa}a^{\kappa-1}B_\kappa}{n+1-\kappa}\sum_{i=1}^{a-1}m_i^{\mu+1-\kappa}\\ 
&=\frac{1}{\mu+1}\sum_{\kappa=0}^{\mu}\binom{\mu+1}{\kappa}B_{\kappa}a^{\kappa-1}\sum_{i=1}^{a_1-1}m_i^{\mu+1-\kappa}\,. 
\end{align*}
If $0<j<\mu+1-\kappa$, by 
\begin{align*}
&\sum_{n=\kappa}^\mu\binom{\mu}{n}\binom{n}{\kappa}\frac{(-1)^{n-\kappa}}{n+1-\kappa}\binom{n+1-\kappa}{j}\\
&=\frac{(-1)^{j-1}\mu!}{j!\kappa!(\mu+1-\kappa-j)!}\sum_{n=0}^{\mu-\kappa}\binom{\mu+1-\kappa-j}{n}(-1)^{\mu+1-\kappa-j-n}\\
&=0\,, 
\end{align*}
all the terms for $m_i^{\mu+1-\kappa-j}(-i)^j$ with $\mu+1-\kappa-j>0$ and $j>0$ are vanished. 
If $n=\mu$ and $j=\mu+1-\kappa$, 
\begin{align}  
&\sum_{\kappa=0}^\mu\binom{\mu}{\kappa}\frac{(-1)^{\mu-\kappa}a^{\kappa-1}B_\kappa}{\mu+1-\kappa}\sum_{i=1}^{a-1}(-i)^{\mu+1-\kappa}\notag\\
&=-\sum_{\kappa=0}^\mu\binom{\mu}{\kappa}\frac{a^{\kappa-1}B_\kappa}{\mu+1-\kappa}\sum_{i=1}^{a-1}i^{\mu+1-\kappa}\notag\\
&=-\sum_{\kappa=0}^\mu\binom{\mu}{\kappa}\frac{a^{\kappa-1}B_\kappa}{\mu+1-\kappa}\sum_{k=0}^{\mu+1-\kappa}\binom{\mu+1-\kappa}{k}B_{\mu+1-\kappa-k}\frac{a^{k+1}}{k+1}\notag\\
&=-\sum_{\kappa=0}^\mu\binom{\mu}{\kappa}\frac{B_\kappa B_{\mu+1-\kappa}}{\mu+1-\kappa}a^\kappa\notag\\
&\quad -\sum_{\kappa=0}^\mu\sum_{k=1}^{\mu+1-\kappa}\binom{\mu}{\kappa}\binom{\mu+1-\kappa}{k}\frac{B_\kappa B_{\mu+1-\kappa-k}}{(\mu+1-\kappa)(k+1)}a^{\kappa+k}\notag\\
&=-\sum_{\ell=0}^\mu\binom{\mu}{\ell}\frac{B_\ell B_{\mu+1-\ell}}{\mu+1-\ell}a^\ell\notag\\
&\quad -\sum_{\ell=1}^{\mu+1}\sum_{\kappa=0}^{\ell-1}\binom{\mu}{\kappa}\binom{\mu+1-\kappa}{\ell-\kappa}\frac{B_\kappa B_{\mu+1-\ell}}{(\mu+1-\kappa)(\ell+1-\kappa)}a^{\ell}\,. 
\label{eq:20}
\end{align}
By using the recurrence relation 
$$
\sum_{\kappa=0}^{\mu+1}\binom{\mu+2}{\kappa}B_\kappa=0\,, 
$$ 
the term of $a^{\mu+1}$ yields from the second sum in (\ref{eq:20}) as 
$$
-\sum_{\kappa=0}^\mu\binom{\mu}{\kappa}\frac{B_\kappa a^{\mu+1}}{(\mu+1-\kappa)(\mu+2-\kappa)}=\frac{B_{\mu+1}}{\mu+1}a^{\mu+1}\,.  
$$   
The constant term yields from the first sum in (\ref{eq:20}) as 
$$
-\frac{B_{\mu+1}}{\mu+1}\,. 
$$ 
Other terms of $a^\ell$ ($1\le\ell\le\mu$) are cancelled, so vanished, because by 
$$
\sum_{\kappa=0}^\ell\binom{\ell+1}{\kappa}B_\kappa=0\,, 
$$ 
we get 
$$
\sum_{\kappa=0}^{\ell-1}\binom{\mu}{\kappa}\binom{\mu+1-\kappa}{\ell-\kappa}\frac{B_\kappa}{(\mu+1-\kappa)(\ell+1-\kappa)}=-\binom{\mu}{\ell}\frac{B_\ell}{\mu+1-\ell}\,.  
$$ 
Hence, we obtain the term 
\begin{align*} 
\sum_{\kappa=0}^{\mu}\binom{\mu}{\kappa}\frac{(-1)^{\mu-\kappa}a^{\kappa-1}B_\kappa}{\mu+1-\kappa}\sum_{i=1}^{a-1}(-i)^{\mu+1-\kappa}
=\frac{B_{\mu+1}}{\mu+1}(a^{\mu+1}-1)\,. 
\end{align*} 
\end{proof}


\subsection{Arithmetic sequences}

Theorem \ref{lem2} is so useful that as long as one can calculate the sum of representable numbers $m_i\equiv i\pmod{a_1}$ ($1\le i\le a_1-1$), we obtain the power sum of nonrepresentable numbers.   In this paper, we apply Theorem \ref{lem2} to obtain the power sum of nonrepresentable numbers representation by arithmetic sequences.   
Let $a$, $d$ and $k$ be positive integers with $\gcd(a,d)=1$ and $2\le k\le a$. 
Roberts \cite{ro56} found the Frobenius number for arithmetic sequences. 
$$
g(a,a+d,\dots,a+(k-1)d)=\fl{\frac{a-2}{k-1}}a+(a-1)d\,.
$$ 
The case $d=1$ had been found by Brauer \cite{br42}.  
Selmer \cite{se77} generalized Roberts' result by giving a formula for almost arithmetic sequences. 

Let $a-1=q(k-1)+r$ with $0\le r<k-1$. Grant \cite{gr73} obtained a formula for the number of positive integers with no nonnegative integer representation by arithmetic sequences.  
$$
n(a,a+d,\dots,a+(k-1)d)=\frac{1}{2}\bigl((a-1)(q+d)+r(q+1)\bigr)\,.
$$ 
Selmer \cite{se77} generalized Grant's result by giving a formula for almost arithmetic sequences. 
Note that $q>0$ because $k\le a$. 
The sum of nonrepresentable numbers in arithmetic progression are given explicitly as follows.

Since the minimal residue system $\{m_i\}$ ($1\le i\le a_1-1$) is given by 

\begin{table}[h]
  \centering
\begin{tabular}{cccccc}
\cline{1-2}\cline{3-4}\cline{5-6}
\multicolumn{1}{|c}{$a_2$}&\multicolumn{1}{|c}{$a_3$}&\multicolumn{1}{|c}{$\cdots$}&\multicolumn{1}{|c}{$a_{r+1}$}&\multicolumn{1}{|c}{$\cdots$}&\multicolumn{1}{|c|}{$a_k$}\\
\cline{1-2}\cline{3-4}\cline{5-6}
\multicolumn{1}{|c}{$a_2+a_k$}&\multicolumn{1}{|c}{$a_3+a_k$}&\multicolumn{1}{|c}{$\cdots$}&\multicolumn{1}{|c}{$a_{r+1}+a_k$}&\multicolumn{1}{|c}{$\cdots$}&\multicolumn{1}{|c|}{$2 a_k$}\\
\cline{1-2}\cline{3-4}\cline{5-6}
\multicolumn{1}{|c}{$\vdots$}&\multicolumn{1}{|c}{$\vdots$}&\multicolumn{1}{|c}{$\vdots$}&\multicolumn{1}{|c}{$\vdots$}&\multicolumn{1}{|c}{$\vdots$}&\multicolumn{1}{|c|}{$\vdots$}\\
\cline{1-2}\cline{3-4}\cline{5-6}
\multicolumn{1}{|c}{$a_2+(q-1)a_k$}&\multicolumn{1}{|c}{$a_3+(q-1)a_k$}&\multicolumn{1}{|c}{$\cdots$}&\multicolumn{1}{|c}{$a_{r+1}+(q-1)a_k$}&\multicolumn{1}{|c}{$\cdots$}&\multicolumn{1}{|c|}{$q a_k$}\\
\cline{1-2}\cline{3-4}\cline{5-6}
\multicolumn{1}{|c}{$a_2+q a_k$}&\multicolumn{1}{|c}{$a_3+q a_k$}&\multicolumn{1}{|c}{$\cdots$}&\multicolumn{1}{|c|}{$a_{r+1}+q a_k$}&&\\
\cline{1-2}\cline{3-4} 
\end{tabular}
\caption{} 
\label{minres}
\end{table} 

\noindent  
(\cite[(3.8)]{se77}), that is, 
\begin{table}[h]
  \centering
\scalebox{0.6}[0.6]{
\begin{tabular}{cccccc}
\cline{1-2}\cline{3-4}\cline{5-6}
\multicolumn{1}{|c}{$a+d$}&\multicolumn{1}{|c}{$a+2 d$}&\multicolumn{1}{|c}{$\cdots$}&\multicolumn{1}{|c}{$a+r d$}&\multicolumn{1}{|c}{$\cdots$}&\multicolumn{1}{|c|}{$a+(k-1)d$}\\
\cline{1-2}\cline{3-4}\cline{5-6}
\multicolumn{1}{|c}{$2 a+k d$}&\multicolumn{1}{|c}{$2 a+(k+1)d$}&\multicolumn{1}{|c}{$\cdots$}&\multicolumn{1}{|c}{$2 a+(k-1+r)d$}&\multicolumn{1}{|c}{$\cdots$}&\multicolumn{1}{|c|}{$2 a+2(k-1)d$}\\
\cline{1-2}\cline{3-4}\cline{5-6}
\multicolumn{1}{|c}{$\vdots$}&\multicolumn{1}{|c}{$\vdots$}&\multicolumn{1}{|c}{$\vdots$}&\multicolumn{1}{|c}{$\vdots$}&\multicolumn{1}{|c}{$\vdots$}&\multicolumn{1}{|c|}{$\vdots$}\\
\cline{1-2}\cline{3-4}\cline{5-6}
\multicolumn{1}{|c}{$q a+((q-1)(k-1)+1)d$}&\multicolumn{1}{|c}{$q a+((q-1)(k-1)+2)d$}&\multicolumn{1}{|c}{$\cdots$}&\multicolumn{1}{|c}{$q a+((q-1)(k-1)+r)d$}&\multicolumn{1}{|c}{$\cdots$}&\multicolumn{1}{|c|}{$q a+q(k-1)d$}\\
\cline{1-2}\cline{3-4}\cline{5-6}
\multicolumn{1}{|c}{$(q+1)a+(q(k-1)+1)d$}&\multicolumn{1}{|c}{$(q+1)a+(q(k-1)+2)d$}&\multicolumn{1}{|c}{$\cdots$}&\multicolumn{1}{|c|}{$(q+1)a+(q(k-1)+r)d$}&&\\
\cline{1-2}\cline{3-4} 
\end{tabular}
} 
\caption{} 
\end{table} 

\noindent 
Here, $q(k-1)+r=a-1$.  
Since $\gcd(a,d)=1$, 
$$
\{d,2d,\dots,(a-1)d\}=\{1,2,\dots,a-1\}\pmod a\,. 
$$ 
Then, the $\nu$-th sum of all the elements is calculated as   
\begin{align*}  
&\sum_{i=1}^{a-1}m_i^\nu\\ 
&=\sum_{l=0}^\nu\binom{\nu}{l}a^{\nu-l}\bigl(d^l+(2 d)^l+\cdots+((k-1)d)^l\bigr)\\  
&\quad +\sum_{l=0}^\nu\binom{\nu}{l}(2 a)^{\nu-l}\bigl((k d)^l+((k+1)d)^l+\cdots+((2 k-2)d)^l\bigr)\\  
&\quad +\cdots\\ 
&\quad +\sum_{l=0}^\nu\binom{\nu}{l}(q a)^{\nu-l}\bigl((((q-1)k-q+2)d)^l\\
&\qquad +(((q-1)k-q+3)d)^l+\cdots+((q k-q)d)^l\bigr)\\   
&\quad +\sum_{l=0}^\nu\binom{\nu}{l}((q+1)a)^{\nu-l}\bigl(((q k-q+1)d)^l+\cdots+((q k-q+r)d)^l
\bigr)\\ 
&=\sum_{l=0}^\nu\binom{\nu}{l}\frac{a^{\nu-l}d^l}{l+1}\sum_{j=0}^l\binom{l+1}{j}B_j(k^{l+1-j}-1)\\
&\quad +\sum_{l=0}^\nu\binom{\nu}{l}\frac{(2 a)^{\nu-l}d^l}{l+1}\sum_{j=0}^l\binom{l+1}{j}B_j\bigl((2 k-1)^{l+1-j}-k^{l+1-j}\bigr)\\ 
&\quad +\cdots\\
&\quad +\sum_{l=0}^\nu\binom{\nu}{l}\frac{(q a)^{\nu-l}d^l}{l+1}\sum_{j=0}^l\binom{l+1}{j}B_j\\
&\qquad\times\bigl((q k-q+1)^{l+1-j}-((q-1)k-q+2)^{l+1-j}\bigr)\\ 
&\quad +\sum_{l=0}^\nu\binom{\nu}{l}\frac{((q+1)a)^{\nu-l}d^l}{l+1}\sum_{j=0}^l\binom{l+1}{j}B_j\\
&\qquad\times\bigl((q k-q+r+1)^{l+1-j}-(q k-q+1)^{l+1-j}\bigr)\\
&=\sum_{l=0}^\nu\binom{\nu}{l}\frac{a^{\nu-l}d^l}{l+1}\sum_{j=0}^l\binom{l+1}{j}B_j\biggl(-1-(2^{\nu-l}-1)k^{l+1-j}\\
&\qquad -(3^{\nu-l}-2^{\nu-l})(2 k-1)^{l+1-j}-\cdots\\
&\qquad -(q^{\nu-l}-(q-1)^{\nu-l})((q-1)k-q+2)^{l+1-j}\\
&\qquad -((q+1)^{\nu-l}-q^{\nu-l})(q k-q+1)^{l+1-j}\\
&\qquad +(q+1)^{\nu-l}(q k-q+r+1)^{l+1-j}
\biggr)\\
&=\sum_{l=0}^\nu\binom{\nu}{l}\frac{a^{\nu-l}d^l}{l+1}\sum_{j=0}^l\binom{l+1}{j}B_j\biggl((q+1)^{\nu-l}(q k-q+r+1)^{l+1-j}\\
&\quad -1-\sum_{i=1}^q\bigl((i+1)^{\nu-l}-i^{\nu-l}\bigr)\bigl(i(k-1)+1\bigr)^{l+1-j} 
\biggr)\,.  
\end{align*}
By applying Theorem \ref{lem2}, we obtain 
\begin{align*}  
&s_\mu(a,a+d,\dots,a+(k-1)d)\\
&=\frac{1}{\mu+1}\sum_{\kappa=0}^{\mu}\binom{\mu+1}{\kappa}B_{\kappa}a^{\kappa-1}\sum_{l=0}^{\mu+1-\kappa}\binom{\mu+1-\kappa}{l}\frac{a^{\mu+1-\kappa-l}d^l}{l+1}\\
&\qquad\times \sum_{j=0}^l\binom{l+1}{j}B_j\biggl((q+1)^{\mu+1-\kappa-l}a^{l+1-j}-1\\
&\qquad\quad -\sum_{i=1}^q\bigl((i+1)^{\mu+1-\kappa-l}-i^{\mu+1-\kappa-l}\bigr)\bigl(i(k-1)+1\bigr)^{l+1-j} 
\biggr)\\  
&\quad +\frac{B_{\mu+1}}{\mu+1}(a^{\mu+1}-1)\\
&=\frac{1}{\mu+1}\sum_{\kappa=0}^{\mu}\sum_{l=0}^{\mu+1-\kappa}\sum_{j=0}^l\binom{\mu+1}{\kappa}\binom{\mu+1-\kappa}{l}\binom{l+1}{j}\frac{B_{\kappa}B_j a^{\mu-l}d^l}{l+1}\\
&\qquad \times \biggl((q+1)^{\mu+1-\kappa-l}a^{l+1-j}-1\\
&\qquad\quad -\sum_{i=1}^q\bigl((i+1)^{\mu+1-\kappa-l}-i^{\mu+1-\kappa-l}\bigr)\bigl(i(k-1)+1\bigr)^{l+1-j} 
\biggr)\\  
&\quad +\frac{B_{\mu+1}}{\mu+1}(a^{\mu+1}-1)\,. 
\end{align*}

Then, by $q=\fl{(a-1)/(k-1)}$, we have an explicit expression of 
the power sum of nonrepresentable numbers in arithmetic progression.

\begin{theorem}  
For positive integers $a$, $d$ and $k$ with $\gcd(a,d)=1$ and $2\le k\le a$, we have 
\begin{align*}  
&s_\mu(a,a+d,\dots,a+(k-1)d)\\
&=\frac{1}{\mu+1}\sum_{\kappa=0}^{\mu}\sum_{l=0}^{\mu+1-\kappa}\sum_{j=0}^l\binom{\mu+1}{\kappa}\binom{\mu+1-\kappa}{l}\binom{l+1}{j}\frac{B_{\kappa}B_j a^{\mu-l}d^l}{l+1}\\
&\qquad \times \biggl(\left(\fl{\frac{a-1}{k-1}}+1\right)^{\mu+1-\kappa-l}a^{l+1-j}-1\\
&\qquad\quad -\sum_{i=1}^{\fl{(a-1)/(k-1)}}\bigl((i+1)^{\mu+1-\kappa-l}-i^{\mu+1-\kappa-l}\bigr)\bigl(i(k-1)+1\bigr)^{l+1-j} 
\biggr)\\  
&\quad +\frac{B_{\mu+1}}{\mu+1}(a^{\mu+1}-1)\,.  
\end{align*}
\label{th:pw-arith}
\end{theorem} 

\noindent 
{\it Remark.}   
Setting $d=b-a$ and $k=2$, Theorem \ref{th:pw-arith} provides a different expression of power sum for two positive integers $a$ and $b$ from R\o dseth's one in \cite{ro94}. 
Setting $d=b-a$, $k=2$ and $\mu=1$, Theorem \ref{th:pw-arith} is reduced to the formula (\ref{brown}).

\subsection{Examples}  

Consider the sequence $13,16,19,22,25$. So, we see that $a=13$, $d=3$, $k=5$, $q=3$ and $r=0$. Then by Theorem \ref{th:pw-arith},  
\begin{align*}  
s_1(13,16,19,22,25)&=894\,,\\ 
s_2(13,16,19,22,25)&=33150\,,\\ 
s_3(13,16,19,22,25)&=1463868\,,\\ 
s_4(13,16,19,22,25)&=71099730\,,\\ 
s_5(13,16,19,22,25)&=3663620844\,,\\ 
s_6(13,16,19,22,25)&=196356363450\,,\\ 
s_7(13,16,19,22,25)&=10815989768148\,. 
\end{align*}
In fact, the power sum of nonrepresentable numbers is expressed as 
\begin{align*}
&1^\mu + 2^\mu + 3^\mu + 4^\mu + 5^\mu + 6^\mu + 7^\mu + 8^\mu + 9^\mu + 10^\mu + 11^\mu\\
& + 12^\mu + 14^\mu + 15^\mu + 17^\mu + 18^\mu + 20^\mu + 21^\mu + 23^\mu + 24^\mu + 27^\mu\\ 
&+ 28^\mu + 30^\mu + 31^\mu + 33^\mu + 34^\mu + 36^\mu + 37^\mu + 40^\mu + 43^\mu + 46^\mu\\
&+ 49^\mu + 53^\mu + 56^\mu + 59^\mu + 62^\mu\,. 
\end{align*}

Consider the sequences 
\begin{align*}  
&{\rm (1)}\quad 25,29,33,37,41,45,49,53,57;\\ 
&{\rm (2)}\quad 25,29,33,37,41,45,49,53,57,61;\\ 
&{\rm (3)}\quad 25,29,33,37,41,45,49,53,57,61,65;\\ 
&{\rm (4)}\quad 25,29,33,37,41,45,49,53,57,61,65,69\,. 
\end{align*} 
They give the same value of Frobenius number $g(A)=146$. However, the values of Sylvester number and Sylvester sum are different. In the case of power sums, by Theorem \ref{th:pw-arith} with $a=25$, $d=4$ and (1) $k=9$, $q=3$, $r=0$; (2) $k=10$, $q=2$, $r=6$; (3) $k=11$, $q=2$, $r=4$; (4) $k=12$, $q=2$, $r=2$, we see that for $\mu=6$ 
\begin{align*}  
s_6(25,29,33,37,41,45,49,53,57)&=64005202245000\,,\\
s_6(25,29,33,37,41,45,49,53,57,61)&=57956823758511\,,\\
s_6(25,29,33,37,41,45,49,53,57,61,65)&=49053091726510\,,\\
s_6(25,29,33,37,41,45,49,53,57,61,65,69)&=36249074667429\,.
\end{align*}

\section{Weighted sums} 

Let $\sts{n}{m}$ denote the Stirling numbers of the second kind, calculated as 
$$
\sts{n}{m}=\frac{1}{m!}\sum_{i=0}^m(-1)^i\binom{m}{i}(m-i)^n\,. 
$$   
Assume that $\lambda^a\ne 1$ and $\lambda^d\ne 1$. Since for nonnegative integer $\nu$ and positive integer $a$ 
\begin{equation}
a^\nu x^a=\sum_{i=0}^\nu\sts{\nu}{i}x^i\frac{d^i}{d x^i}x^a\,, 
\label{eq:432} 
\end{equation}  
the weighted sum of the first line in Table \ref{minres} is equal to 
\begin{align*}  
&a_2^\nu\lambda^{a_2}+a_3^\nu\lambda^{a_3}+\cdots+a_k^\nu\lambda^{a_k}\\
&=\sum_{i=0}^\nu\sts{\nu}{i}\lambda^i\left[\frac{d^i}{d x^i}(x^{a_2}+x^{a_3}+\cdots+x^{a_k})\right]_{x=\lambda}\\ 
&=\sum_{i=0}^\nu\sts{\nu}{i}\lambda^i\left[\frac{d^i}{d x^i}(x^{a+d}+x^{a+2 d}+\cdots+x^{a+(k-1)d})\right]_{x=\lambda}\\ 
&=\sum_{i=0}^\nu\sts{\nu}{i}\lambda^i\left[\frac{d^i}{d x^i}\left(x^{a+d}\frac{x^{(k-1)d}-1}{x^d-1}\right)\right]_{x=\lambda}\,.  
\end{align*} 
Similarly, the weighted sum of the second line in Table \ref{minres} is equal to 
\begin{align*}  
&\sum_{i=0}^\nu\sts{\nu}{i}\lambda^i\left[\frac{d^i}{d x^i}(x^{2 a+k d}+x^{2 a+(k+1)d}+\cdots+x^{2 a+2(k-1)d})\right]_{x=\lambda}\\ 
&=\sum_{i=0}^\nu\sts{\nu}{i}\lambda^i\left[\frac{d^i}{d x^i}\left(x^{2 a+k d}\frac{x^{(k-1)d}-1}{x^d-1}\right)\right]_{x=\lambda}\,.  
\end{align*} 
The weighted sums of the third, and eventually the $q$-th line in Table \ref{minres} are given by 
\begin{align*} 
&\sum_{i=0}^\nu\sts{\nu}{i}\lambda^i\left[\frac{d^i}{d x^i}\left(x^{3 a+(2 k-1)d}\frac{x^{(k-1)d}-1}{x^d-1}\right)\right]_{x=\lambda}\,,\\
&\sum_{i=0}^\nu\sts{\nu}{i}\lambda^i\left[\frac{d^i}{d x^i}\left(x^{q a+((q-1)k-q+2)d}\frac{x^{(k-1)d}-1}{x^d-1}\right)\right]_{x=\lambda}\,,    
\end{align*}  
respectively.  Finally, if $r>0$, the weighted sum of the final line is given by\begin{align*}  
&\sum_{i=0}^\nu\sts{\nu}{i}\lambda^i\left[\frac{d^i}{d x^i}(x^{(q+1)a+(q k-q+1)d}+\cdots+x^{(q+1)a+(q k-q+r)d})\right]_{x=\lambda}\\ 
&=\sum_{i=0}^\nu\sts{\nu}{i}\lambda^i\left[\frac{d^i}{d x^i}\left(x^{(q+1)a+(q k-q+1)d}\frac{x^{r d}-1}{x^d-1}\right)\right]_{x=\lambda}\,.  
\end{align*}  
Hence, for $\nu\ge 1$, we obtain that 
\begin{align*} 
&\sum_{i=0}^{a-1}\lambda^{m_i}m_i^\nu=\sum_{i=1}^{a-1}\lambda^{m_i}m_i^\nu\\
&=\sum_{i=0}^\nu\sts{\nu}{i}\lambda^i\biggl[\frac{d^i}{d x^i}\biggl(
\bigl(x^{a+d}+x^{2 a+k d}+x^{3 a+(2 k-1)d}+\cdots\\
&\quad   +x^{q a+((q-1)k-q+2)d}\bigr)\frac{x^{(k-1)d}-1}{x^d-1} 
 +x^{(q+1)a+(q k-q+1)d}\frac{x^{r d}-1}{x^d-1} 
\biggr)\biggr]_{x=\lambda}\\ 
&=\sum_{i=0}^\nu\sts{\nu}{i}\lambda^i\biggl[\frac{d^i}{d x^i}\biggl( 
x^{a+d}\frac{x^{q a_k}-1}{x^{a_k}-1}\frac{x^{(k-1)d}-1}{x^d-1}
+x^{(q+1)a+(q k-q+1)d}\frac{x^{r d}-1}{x^d-1} 
\biggr)\biggr]_{x=\lambda}\\ 
&=\sum_{i=0}^\nu\sts{\nu}{i}\lambda^i\biggl[\frac{d^i}{d x^i}\biggl( 
\frac{x^{d}(x^{a_k}-x^{a})}{x^d-1}\frac{x^{q a_k}-1}{x^{a_k}-1}
+\frac{x^{q a_k+a+d}(x^{r d}-1)}{x^d-1} 
\biggr)\biggr]_{x=\lambda}\,.  
\end{align*} 
For $\nu=0$, by adding the term $\lambda^{m_0}=1$, we have 
$$
\sum_{i=0}^{a-1}\lambda^{m_i}=1+ 
\frac{\lambda^{d}(\lambda^{a_k}-\lambda^{a})}{\lambda^d-1}\frac{\lambda^{q a_k}-1}{\lambda^{a_k}-1}
+\frac{\lambda^{q a_k+a+d}(\lambda^{r d}-1)}{\lambda^d-1}\,. 
$$ 
Therefore, for $\nu\ge 0$, 
\begin{multline}
\sum_{i=0}^{a-1}\lambda^{m_i}m_i^\nu\\
=\sum_{i=0}^\nu\sts{\nu}{i}\lambda^i\biggl[\frac{d^i}{d x^i}\biggl(1+ 
\frac{x^{d}(x^{a_k}-x^{a})}{x^d-1}\frac{x^{q a_k}-1}{x^{a_k}-1}
+\frac{x^{q a_k+a+d}(x^{r d}-1)}{x^d-1} 
\biggr)\biggr]_{x=\lambda}\,.  
\label{eq:lammu}
\end{multline}

Now, we apply the following formula in \cite{KZ} to obtain the weighted sum of nonrepresentable numbers in arithmetic progression. Here, Eulerian numbers $\eul{n}{m}$ appear in the generating function
\begin{equation}
\sum_{k=0}^\infty k^n x^k=\frac{1}{(1-x)^{n+1}}\sum_{m=0}^{n-1}\eul{n}{m}x^{m+1}\quad(n\ge 1)
\label{eu:gf}
\end{equation} 
with $0^0=1$ and $\eul{0}{0}=1$ (\cite[p.244]{com74}), 
and have an explicit formula 
$$
\eul{n}{m}=\sum_{k=0}^{m}(-1)^k\binom{n+1}{k}(m-k+1)^n
$$   
(\cite[p.243]{com74},\cite{gkp89}).  Similarly, for $1\le i\le a_1-1$, let $m_i$ be the least positive integer congruent to $i\pmod{a_1}$ in ${\rm R}(A)$, where $A=\{a_1,a_2,\dots,a_k\}$. For convenience, put $m_0=0$.    

\begin{Lem}[{\rm \cite{KZ}}]  
Assume that $\lambda\ne 0$ and $\lambda^{a_1}\ne 1$. Then for a positive integer $\mu$,  
\begin{align*}  
s_\mu^{(\lambda)}(A)&=\sum_{n=0}^\mu\frac{(-a_1)^n}{(\lambda^{a_1}-1)^{n+1}}\binom{\mu}{n}\sum_{j=0}^n\eul{n}{n-j}\lambda^{j a_1}\sum_{i=0}^{a_1-1}m_i^{\mu-n}\lambda^{m_i}\\
&\quad +\frac{(-1)^{\mu+1}}{(\lambda-1)^{\mu+1}}\sum_{j=0}^\mu\eul{\mu}{\mu-j}\lambda^j\,.
\end{align*}
\label{lem-hh}
\end{Lem} 

\noindent 
{\it Remark.}  
If one wants to avoid $0^0=1$ (e.g., in computational calculations), we use the formula  
\begin{align*}  
s_\mu^{(\lambda)}(A)&=\sum_{n=0}^{\mu-1}\frac{(-a_1)^n}{(\lambda^{a_1}-1)^{n+1}}\binom{\mu}{n}\sum_{j=0}^n\eul{n}{n-j}\lambda^{j a_1}\sum_{i=0}^{a_1-1}m_i^{\mu-n}\lambda^{m_i}\\
&\quad +\frac{(-a_1)^\mu}{(\lambda^{a_1}-1)^{\mu+1}}\sum_{j=0}^\mu\eul{\mu}{\mu-j}\lambda^{j a_1}\sum_{i=0}^{a_1-1}\lambda^{m_i}\\ 
&\quad +\frac{(-1)^{\mu+1}}{(\lambda-1)^{\mu+1}}\sum_{j=0}^\mu\eul{\mu}{\mu-j}\lambda^j\,.
\end{align*} 
Note that for $k=2$, we have simpler formulas in terms of the Apostol-Bernoulli numbers \cite{KZ0}.

In conclusion, we have a formula for weighted sum of nonrepresentable numbers in arithmetic progression.

\begin{theorem}
Let $a,d,k$ be positive integers with $\gcd(a,d)=1$ and $2\le k\le a$. Let $q$ and $r$ be nonnegative integers with $a-1=q(k-1)+r$ and $0\le r<k-1$. 
Assume that $\lambda\ne 0$ and $\lambda^{a}\ne 1$ and $\lambda^d\ne 1$. Then for a positive integer $\mu$,  
\begin{align*}  
&s_\mu^{(\lambda)}(a,a+d,\dots,a+(k-1)d)\\
&=\sum_{n=0}^\mu\frac{(-a)^n}{(\lambda^{a}-1)^{n+1}}\binom{\mu}{n}\sum_{j=0}^n\eul{n}{n-j}\lambda^{j a}\\
&\qquad\times\sum_{i=0}^{\mu-n}\sts{\mu-n}{i}\lambda^i\biggl[\frac{d^i}{d x^i}\biggl(1+ 
\frac{x^{d}(x^{a_k}-x^{a})}{x^d-1}\frac{x^{q a_k}-1}{x^{a_k}-1}
+\frac{x^{q a_k+a+d}(x^{r d}-1)}{x^d-1} 
\biggr)\biggr]_{x=\lambda}\\
&\quad +\frac{(-1)^{\mu+1}}{(\lambda-1)^{\mu+1}}\sum_{j=0}^\mu\eul{\mu}{\mu-j}\lambda^j\,.
\end{align*}
\label{th:wpw-arith}
\end{theorem}

\subsection{Examples}  

Consider the sequence $14, 17, 20, 23, 26, 29$.  Then, $a=14$, $d=3$, $k=6$, $q=2$ and $r=3$. By Theorem \ref{th:wpw-arith}, we have 
\begin{align*}
&s_2^{(\sqrt[3]{2})}(14,17,20,23,26,29)\\
&=21528522+31320173525\sqrt[3]{2}+659369214\sqrt[3]{4}\,,\\ 
&s_3^{(7)}(14,17,20,23,26,29)\\
&=126153136547718860397749189364814847897329040723302499959511892\,,\\
&s_4^{(-1/2)}(14,17,20,23,26,29)\\
&=-\frac{252455039549405466513}{147573952589676412928}\,,\\
&s_5^{(4+3\sqrt{-1})}(14,17,20,23,26,29)\\
&=58604955584641578954030966530484875253297329000101560480\\
&\quad -69984733631939902694215153740002368436325991046609895240\sqrt{-1}\,. 
\end{align*}
In fact, the weight power sum of nonrepresentable numbers is given by 
\begin{align*}  
&\lambda^1\cdot 1^\mu + \lambda^2\cdot 2^\mu + \lambda^3\cdot 3^\mu + \lambda^4\cdot 4^\mu + \lambda^5\cdot 5^\mu+\lambda^6\cdot 6^\mu + \lambda^7\cdot 7^\mu+ \lambda^8\cdot 8^\mu  \\ 
&+ \lambda^9\cdot 9^\mu + \lambda^{10}\cdot 10^\mu 
+\lambda^{11}\cdot 11^\mu + \lambda^{12}\cdot 12^\mu + \lambda^{13}\cdot 13^\mu + \lambda^{15}\cdot 15^\mu\\ 
&+\lambda^{16}\cdot 16^\mu + \lambda^{18}\cdot 18^\mu + \lambda^{19}\cdot 19^\mu + \lambda^{21}\cdot 21^\mu+\lambda^{22}\cdot 22^\mu + \lambda^{24}\cdot 24^\mu \\ 
&+ \lambda^{25}\cdot 25^\mu + \lambda^{27}\cdot 27^\mu 
+\lambda^{30}\cdot 30^\mu + \lambda^{32}\cdot 32^\mu + \lambda^{33}\cdot 33^\mu + \lambda^{35}\cdot 35^\mu\\ 
&+\lambda^{36}\cdot 36^\mu + \lambda^{38}\cdot 38^\mu + \lambda^{39}\cdot 39^\mu + \lambda^{41}\cdot 41^\mu+\lambda^{44}\cdot 44^\mu + \lambda^{47}\cdot 47^\mu \\ 
&+ \lambda^{50}\cdot 50^\mu + \lambda^{53}\cdot 53^\mu +\lambda^{61}\cdot 61^\mu + \lambda^{64}\cdot 64^\mu + \lambda^{67}\cdot 67^\mu\,. 
\end{align*}

\section{Weighted sums in exceptional cases}

Theorem \ref{th:wpw-arith} does not include the so-called alternate sums (e.g., see, \cite{wang08}). For example, we cannot calculate $s_1^{(-1)}(14,17,20,23,26,29)$ directly by Theorem \ref{th:wpw-arith}, because $\lambda^a=(-1)^{14}=1$. It should be done as $-1+2-3+4-5+6-7+8-9+10-11+12-13-15+16+\cdots-67$.  In this section, we study such exceptional cases. We need major modifications.

\subsection{The case $\lambda^d=1$} 

Since $\gcd(a,d)=1$, it is impossible for both $\lambda^a=1$ and $\lambda^d=1$ whenever $\lambda\ne 1$. First, assume that $\lambda^a\ne 1$ and $\lambda^d=1$. Then, the weighted sum of the first line of Table \ref{minres} is given by 
\begin{align*} 
&\lambda^a\bigl((a+d)^\nu+(a+2 d)^\nu+\cdots+(a+(k-1)d)^\nu\bigr)\\
&=\lambda^a\sum_{j=1}^{k-1}\sum_{\ell=0}^\nu\binom{\nu}{\ell}a^{\nu-\ell}(j d)^\ell\\ 
&=\lambda^a\sum_{\ell=0}^\nu\binom{\nu}{\ell}\frac{a^{\nu-\ell}d^\ell}{\ell+1}\sum_{j=0}^\ell\binom{\ell+1}{j}B_j(k^{\ell+1-j}-1)\,. 
\end{align*}
That of the second line is given by 
\begin{align*}  
&\lambda^{2 a}\bigl((2 a+k d)^\nu+(2 a+(k+1)d)^\nu+\cdots+(a+2(k-1)d)^\nu\bigr)\\ 
&=\lambda^{2 a}\sum_{\ell=0}^\nu\binom{\nu}{\ell}\frac{(2 a)^{\nu-\ell}d^\ell}{\ell+1}\sum_{j=0}^\ell\binom{\ell+1}{j}B_j\bigl((2 k-1)^{\ell+1-j}-k^{\ell+1-j}\bigr)\,.
\end{align*}
Similarly, those of the third and eventually the $q$-th line are given by 
$$ 
\lambda^{3 a}\sum_{\ell=0}^\nu\binom{\nu}{\ell}\frac{(3 a)^{\nu-\ell}d^\ell}{\ell+1}\sum_{j=0}^\ell\binom{\ell+1}{j}B_j\bigl((3 k-2)^{\ell+1-j}-(2 k-1)^{\ell+1-j}\bigr)
$$ 
and 
\begin{align*}  
&\lambda^{q a}\sum_{\ell=0}^\nu\binom{\nu}{\ell}\frac{(q a)^{\nu-\ell}d^\ell}{\ell+1}\sum_{j=0}^\ell\binom{\ell+1}{j}B_j\\
&\qquad\times\bigl((q(k-1)+1)^{\ell+1-j}-((q-1)(k-1)+1)^{\ell+1-j}\bigr)\,,\\
\end{align*}
respectively. Finally, the weighted sum of the ($q+1$)-st line is given by 
\begin{align*}  
&\lambda^{(q+1)a}\sum_{\ell=0}^\nu\binom{\nu}{\ell}\frac{((q+1)a)^{\nu-\ell}d^\ell}{\ell+1}\sum_{j=0}^\ell\binom{\ell+1}{j}B_j\\
&\qquad\times\bigl((q(k-1)+r+1)^{\ell+1-j}-(q(k-1)+1)^{\ell+1-j}\bigr)\,. 
\end{align*} 
Thus, 
\begin{align*}  
&\sum_{i=0}^{a-1}\lambda^{m_i}m_i^\nu\\
&=-\sum_{\ell=0}^\nu\binom{\nu}{\ell}\frac{\lambda^a a^{\nu-\ell}d^\ell}{\ell+1}\sum_{j=0}^\ell\binom{\ell+1}{j}B_j\\
&\quad -\sum_{\ell=0}^\nu\binom{\nu}{\ell}\frac{\bigl(\lambda^{2 a}(2 a)^{\nu-\ell}-\lambda^a a^{\nu-\ell}\bigr)d^\ell}{\ell+1}\sum_{j=0}^\ell\binom{\ell+1}{j}B_j k^{\ell+1-j}\\ 
&\quad -\sum_{\ell=0}^\nu\binom{\nu}{\ell}\frac{\bigl(\lambda^{3 a}(3 a)^{\nu-\ell}-\lambda^{2 a}(2 a)^{\nu-\ell}\bigr)d^\ell}{\ell+1}\sum_{j=0}^\ell\binom{\ell+1}{j}B_j(2 k-1)^{\ell+1-j}\\ 
&\quad -\cdots\\ 
&\quad -\sum_{\ell=0}^\nu\binom{\nu}{\ell}\frac{\bigl(\lambda^{(q+1)a}((q+1)a)^{\nu-\ell}-\lambda^{q a}(q a)^{\nu-\ell}\bigr)d^\ell}{\ell+1}\\
&\qquad\times\sum_{j=0}^\ell\binom{\ell+1}{j}B_j\bigl(q(k-1)+1\bigr)^{\ell+1-j}\\ 
&\quad +\sum_{\ell=0}^\nu\binom{\nu}{\ell}\frac{\lambda^{(q+1)a}((q+1)a)^{\nu-\ell}d^\ell}{\ell+1}\sum_{j=0}^\ell\binom{\ell+1}{j}B_j\bigl(q(k-1)+r+1\bigr)^{\ell+1-j}\\ 
&=\sum_{\ell=0}^\nu\binom{\nu}{\ell}\frac{d^\ell}{\ell+1}\sum_{j=0}^\ell\binom{\ell+1}{j}B_j\biggl(-\lambda^a a^{\nu-\ell}\\
&\quad -\sum_{s=1}^q\bigl(\lambda^{(s+1)a}((s+1)a)^{\nu-\ell}-\lambda^{s a}(s a)^{\nu-\ell}\bigr)\bigl(s(k-1)+1\bigr)^{\ell+1-j}\\
&\quad +\lambda^{(q+1)a}\bigl((q+1)a\bigr)^{\nu-\ell}\bigl(q(k-1)+r+1\bigr)^{\ell+1-j}  
\biggr)\,. 
\end{align*}

By applying Lemma \ref{lem-hh}, we have an explicit expression of the weighted sum of nonrepresentable numbers for $\lambda^d=1$.

\begin{theorem}
Let $a,d,k$ be positive integers with $\gcd(a,d)=1$ and $2\le k\le a$. Let $q$ and $r$ be nonnegative integers with $a-1=q(k-1)+r$ and $0\le r<k-1$. 
Assume that $\lambda\ne 0$ and $\lambda^{a}\ne 1$ and $\lambda^d=1$. Then for a positive integer $\mu$,  
\begin{align*}  
&s_\mu^{(\lambda)}(a,a+d,\dots,a+(k-1)d)\\
&=\sum_{n=0}^\mu\frac{(-a)^n}{(\lambda^{a}-1)^{n+1}}\binom{\mu}{n}\sum_{j=0}^n\eul{n}{n-j}\lambda^{j a}\\
&\qquad\times\sum_{\ell=0}^{\mu-n}\binom{\mu-n}{\ell}\frac{d^\ell}{\ell+1}\sum_{j=0}^\ell\binom{\ell+1}{j}B_j\biggl(-\lambda^a a^{\mu-n-\ell}\\
&\quad -\sum_{s=1}^q\bigl(\lambda^{(s+1)a}((s+1)a)^{\mu-n-\ell}-\lambda^{s a}(s a)^{\mu-n-\ell}\bigr)\bigl(s(k-1)+1\bigr)^{\ell+1-j}\\
&\quad +\lambda^{(q+1)a}\bigl((q+1)a\bigr)^{\mu-n-\ell}\bigl(q(k-1)+r+1\bigr)^{\ell+1-j}  
\biggr)\\
&\quad +\frac{(-1)^{\mu+1}}{(\lambda-1)^{\mu+1}}\sum_{j=0}^\mu\eul{\mu}{\mu-j}\lambda^j\,.
\end{align*}
\label{th:wpw-arith-d1}
\end{theorem}

\subsection{The case $\lambda^a=1$} 

Next, assume that $\lambda^a=1$ and $\lambda^d\ne 1$. 
In this case, we cannot use Lemma \ref{lem-hh}. We need another formula.  

For $\ell_i=(m_i-i)/a$, 
\begin{align*}  
&\sum_{i=1}^{a-1}\sum_{j=1}^{\ell_i}\lambda^{m_i-j a}(m_i-j a)^\mu\\
&=\sum_{i=1}^{a-1}\sum_{j=1}^{\ell_i}\lambda^{m_i}\sum_{n=0}^\mu\binom{\mu}{n}m_i^{\mu-n}j^n(-a)^n\\ 
&=\sum_{i=1}^{a-1}\lambda^{m_i}\sum_{n=0}^\mu\binom{\mu}{n}m_i^{\mu-n}(-a)^n\sum_{\kappa=0}^n\binom{n}{\kappa}\frac{(-1)^\kappa B_\kappa}{n+1-\kappa}\left(\frac{m_i-i}{a}\right)^{n+1-\kappa}\\ 
&=\sum_{i=1}^{a-1}\lambda^{m_i}\sum_{n=0}^\mu\binom{\mu}{n}m_i^{\mu-n}(-a)^n\sum_{\kappa=0}^n\binom{n}{\kappa}\frac{(-1)^\kappa B_\kappa}{n+1-\kappa}\\
&\qquad\times \sum_{j=0}^{n+1-\kappa}\binom{n+1-\kappa}{j}\frac{m_i^{n+1-\kappa-j}(-i)^j}{a^{n+1-\kappa}}\\ 
&=\sum_{\kappa=0}^\mu B_\kappa a^{\kappa-1}\sum_{n=\kappa}^\mu\binom{\mu}{n}\binom{n}{\kappa}\frac{(-1)^{n-\kappa}}{n+1-\kappa}\\
&\qquad\times \sum_{j=0}^{n+1-\kappa}(-1)^j\binom{n+1-\kappa}{j}\sum_{i=1}^{a-1}m_i^{\mu+1-\kappa-j}\lambda^{m_i}i^j\,. 
\end{align*}
For $j=0$, we have 
\begin{align*}  
\sum_{n=\kappa}^\mu\binom{\mu}{n}\binom{n}{\kappa}\frac{(-1)^{n-\kappa}}{n+1-\kappa}
&=\frac{1}{\mu+1}\binom{\mu+1}{\kappa}\sum_{n=\kappa}^\mu\binom{\mu+1-\kappa}{n+1-\kappa}(-1)^{n-\kappa}\\ 
&=\frac{1}{\mu+1}\binom{\mu+1}{\kappa}\left(1-\sum_{m=0}^{\mu+1-\kappa}\binom{\mu+1-\kappa}{m}(-1)^m\right)\\
&=\frac{1}{\mu+1}\binom{\mu+1}{\kappa}\,. 
\end{align*}
For $1\le j\le n+1-\kappa\le\mu+1-\kappa$, by 
$$  
\sum_{n=\kappa}^\mu(-1)^{n-\kappa}\binom{\mu+1-\kappa}{n+1-\kappa}\binom{n+1-\kappa}{j}=\begin{cases}
(-1)^{\mu-\kappa}&\text{if $j=\mu-\kappa+1$ (and $n=\mu$)};\\ 
0&\text{otherwise}\,, 
\end{cases}
$$  
we have 
\begin{align*}
&\sum_{n=\kappa}^\mu\binom{\mu}{n}\binom{n}{\kappa}\frac{(-1)^{n-\kappa}}{n+1-\kappa}\sum_{j=1}^{n+1-\kappa}(-1)^j\binom{n+1-\kappa}{j}m_i^{\mu+1-\kappa-j}i^j\\ 
&=\frac{1}{\mu+1}\binom{\mu+1}{\kappa}\sum_{n=\kappa}^\mu\binom{\mu+1-\kappa}{n+1-\kappa}(-1)^{n-\kappa}\sum_{j=1}^{n+1-\kappa}(-1)^j\binom{n+1-\kappa}{j}m_i^{\mu+1-\kappa-j}i^j\\
&=\frac{-1}{\mu+1}\binom{\mu+1}{\kappa}i^{\mu-\kappa+1}\,. 
\end{align*} 
Thus, 
\begin{align*}  
&s_\mu^{(\lambda)}(a,a+d,\dots,a+(k-1)d)
=\sum_{i=0}^{a-1}\sum_{j=1}^{\ell_i}\lambda^{m_i-j a}(m_i-j a)^\mu\\
&=\frac{1}{\mu+1}\sum_{\kappa=0}^\mu\binom{\mu+1}{\kappa}B_\kappa a^{\kappa-1}\sum_{i=0}^{a-1}(m_i^{\mu+1-\kappa}-i^{\mu+1-\kappa})\lambda^{m_i}\,. 
\end{align*}
Finally, since $\lambda^{m_i}=\lambda^{a\ell_i+i}=\lambda^i$ and (\ref{eu:gf}), we get 
\begin{align*}
&-\frac{1}{\mu+1}\sum_{\kappa=0}^\mu\binom{\mu+1}{\kappa}B_\kappa a^{\kappa-1}\sum_{i=0}^{a-1}i^{\mu+1-\kappa}\lambda^{m_i}\\
&=-\frac{1}{\mu+1}\sum_{\kappa=0}^\mu\binom{\mu+1}{\kappa}B_\kappa a^{\kappa-1}\sum_{i=0}^{a-1}i^{\mu+1-\kappa}\lambda^{i}\\
&=-\frac{1}{\mu+1}\sum_{i=1}^{a-1}i^\mu\lambda^i\sum_{\kappa=0}^\mu\binom{\mu+1}{\kappa}B_\kappa\left(\frac{a}{i}\right)\\
&=-\sum_{i=0}^{a-1}\sum_{j=1}^\infty\lambda^{i-j a}(i-j a)^\mu\\ 
&=-\sum_{k=1}^\infty(-k)^\mu\lambda^{-k}\\
&=\frac{(-1)^{\mu+1}}{(\lambda-1)^{\mu+1}}\sum_{j=0}^{\mu-1}\eul{\mu}{j}\lambda^{j+1}\,. 
\end{align*}

\begin{theorem} 
Assume that $\lambda\ne 0$ and $\lambda^{a}=1$. Then for a positive integer $\mu$,  
\begin{align}  
s_\mu^{(\lambda)}(a,a+d,\dots,a+(k-1)d)&=\frac{1}{\mu+1}\sum_{n=0}^\mu\binom{\mu+1}{n}B_n a^{n-1}\sum_{i=1}^{a-1}(m_i^{\mu+1-n}-i^{\mu+1-n})\lambda^{m_i}
\label{lem-hh-a1-1}\\
&=\frac{1}{\mu+1}\sum_{n=0}^\mu\binom{\mu+1}{n}B_n a^{n-1}\sum_{i=1}^{a-1}m_i^{\mu+1-n}\lambda^{m_i}\notag\\
&\quad +\frac{(-1)^{\mu+1}}{(\lambda-1)^{\mu+1}}\sum_{j=0}^{\mu-1}\eul{\mu}{j}\lambda^{j+1}\,. 
\label{lem-hh-a1-2}
\end{align}
\label{lem-hh-a1} 
\end{theorem} 

Theorem \ref{lem-hh-a1} is the counterpart version of Lemma \ref{lem-hh}.  
The formula (\ref{lem-hh-a1-1}) looks nicer, but as the correspondence between $m_i$ and $i$ is not so simple, the calculation of $i^{\mu+1-n}\lambda^{m_i}$ is not straightforward. See examples below.  
\bigskip

The formula (\ref{lem-hh-a1-2}) needs an expression of weighted sum of representable numbers in the minimal residue system.  
Since 
\begin{align*} 
D_j&=D_{j,\ell,d,\lambda}:=\lambda^{j d}j^\ell\\
&=\left[\sum_{h=0}^\ell\sts{\ell}{h}x^h\frac{d^h}{d x^h}x^j\right]_{x=\lambda^d}
\end{align*}  
by (\ref{eq:432}), 
the weighted sum of the first line of Table \ref{minres} is given by 
\begin{align*} 
&\lambda^d(a+d)^\nu+\lambda^{2 d}(a+2 d)^\nu+\cdots+\lambda^{(k-1)d}(a+(k-1)d)^\nu\\
&=\sum_{\ell=0}^\nu\binom{\nu}{\ell}a^{\nu-\ell}d^\ell\sum_{j=1}^{k-1}(\lambda^d)^j j^\ell\\ 
&=\sum_{\ell=0}^\nu\binom{\nu}{\ell}a^{\nu-\ell}d^\ell\sum_{j=1}^{k-1}D_j\,. 
\end{align*}
That of the second line is given by 
\begin{align*} 
&\lambda^{k d}(2 a+k d)^\nu+\lambda^{(k+1)d}(2 a+(k+1)d)^\nu+\cdots+\lambda^{2(k-1)d}(2 a+2(k-1)d)^\nu\\
&=\sum_{\ell=0}^\nu\binom{\nu}{\ell}(2 a)^{\nu-\ell}d^\ell\sum_{j=k}^{2(k-1)}D_j\,. 
\end{align*}
Similarly, those of the third line and eventually the $q$-th line are given by 
$$ 
\sum_{\ell=0}^\nu\binom{\nu}{\ell}(3 a)^{\nu-\ell}d^\ell\sum_{j=2 k-1}^{3(k-1)}D_j
$$
and 
$$ 
\sum_{\ell=0}^\nu\binom{\nu}{\ell}(q a)^{\nu-\ell}d^\ell\sum_{j=(q-1)(k-1)+1}^{q(k-1)}D_j\,,
$$ 
respectively.  Finally, that of the ($q+1$)-st line is given by 
$$
\sum_{\ell=0}^\nu\binom{\nu}{\ell}\bigl((q+1)a\bigr)^{\nu-\ell}d^\ell\sum_{j=q(k-1)+1}^{q(k-1)+r}D_j\,. 
$$ 
Thus, we have 
\begin{align*}  
&\sum_{i=1}^{a-1}\lambda^{m_i}m_i^\nu=\sum_{\ell=0}^\nu\binom{\nu}{\ell}d^\ell\\
&\qquad\times\biggl(\sum_{s=1}^q(s a)^{\nu-\ell}\sum_{j=(s-1)(k-1)+1}^{s(k-1)}D_j+\bigl((q+1)a\bigr)^{\nu-\ell}\sum_{j=q(k-1)+1}^{q(k-1)+r}D_j
\biggr)\,. 
\end{align*}

By applying the formula (\ref{lem-hh-a1-2}) in Theorem \ref{lem-hh-a1}, we have an explicit expression of the weighted sum of nonrepresentable numbers for $\lambda^a=1$.  

\begin{theorem} 
Let $a,d,k$ be positive integers with $\gcd(a,d)=1$ and $2\le k\le a$. Let $q$ and $r$ be nonnegative integers with $a-1=q(k-1)+r$ and $0\le r<k-1$. 
Assume that $\lambda\ne 0$ and $\lambda^{a}=1$ and $\lambda^d\ne 1$. Then for a positive integer $\mu$,  
\begin{align*}  
&s_\mu^{(\lambda)}(a,a+d,\dots,a+(k-1)d)\\
&=\frac{1}{\mu+1}\sum_{n=0}^\mu\binom{\mu+1}{n}B_n a^{n-1}\sum_{\ell=0}^{\mu+1-n}\binom{\mu+1-n}{\ell}d^\ell\\
&\qquad\times\biggl(\sum_{s=1}^q(s a)^{\mu+1-n-\ell}\sum_{j=(s-1)(k-1)+1}^{s(k-1)}D_j+\bigl((q+1)a\bigr)^{\mu+1-n-\ell}\sum_{j=q(k-1)+1}^{q(k-1)+r}D_j
\biggr)\\
&\quad +\frac{(-1)^{\mu+1}}{(\lambda-1)^{\mu+1}}\sum_{j=0}^{\mu-1}\eul{\mu}{j}\lambda^{j+1}\,, 
\end{align*}
where 
$$
D_j=\lambda^{j d}j^\ell
=\left[\sum_{h=0}^\ell\sts{\ell}{h}x^h\frac{d^h}{d x^h}x^j\right]_{x=\lambda^d}\,.
$$ 
\label{th:wpw-arith-a1}
\end{theorem}

\subsection{Examples}   

Consider the sequence $12,17,22,27,32,37,42$. So, $a=12$, $d=5$, $k=7$, $q=1$ and $r=5$. If 
$$
\lambda=\zeta_5=e^{2\pi\sqrt{-1}/5}=\frac{-1+\sqrt{5}}{4}+\sqrt{-\frac{5+\sqrt{5}}{8}}
$$  
so that $\lambda^5=1$, 
then by applying Theorem \ref{th:wpw-arith-d1} we have 
\begin{align*}  
s_1^{(\zeta_5)}&=\frac{1}{8}\Bigl(1322-686\sqrt{5}-45\sqrt{-2(5+\sqrt{5})}+87\sqrt{-10(5+\sqrt{5})}\Bigr)\,,\\ 
s_2^{(\zeta_5)}&=\frac{1}{8}\Bigl(49290-39326\sqrt{5}-6611\sqrt{-2(5+\sqrt{5})}+10287\sqrt{-10(5+\sqrt{5})}\Bigr)\,,\\ 
s_3^{(\zeta_5)}&=\frac{1}{8}\Bigl(1846526-2218802\sqrt{5}-652311\sqrt{-2(5+\sqrt{5})}\\
&\quad +756063\sqrt{-10(5+\sqrt{5})}\Bigr)\,,\\ 
s_4^{(\zeta_5)}&=\frac{1}{8}\Bigl(65407506-127134470\sqrt{5}-46993307\sqrt{-2(5+\sqrt{5})}\\
&\quad +49955103\sqrt{-10(5+\sqrt{5})}\Bigr)\,,\\ 
s_5^{(\zeta_5)}&=\frac{1}{8}\Bigl(1911457022-7439898866\sqrt{5}-3100720215\sqrt{-2(5+\sqrt{5})}\\
&\quad +3186196287\sqrt{-10(5+\sqrt{5})}\Bigr)\,. 
\end{align*}

Re-consider the sequence $14,17,20,23,26,29$. So, $a=14$. For $i=1,2,3,\dots,13$, we see that the minimum residue system is 
\begin{align*}
&17&\,&20&\,&23&\,&26&\,&29\\
&46&\,&49&\,&52&\,&55&\,&58\\
&75&\,&78&\,&81&\,&&\,&
\end{align*} 
and taking modulo $14$, it corresponds to 
\begin{align*}
&3&\,&6&\,&9&\,&12&\,&1\\
&4&\,&7&\,&10&\,&13&\,&2\\
&5&\,&8&\,&11&\,&&\,&
\end{align*} 
Therefore, by the formula (\ref{lem-hh-a1-1}) in Theorem \ref{lem-hh-a1}, 
\begin{align*}
&s_\mu^{(-1)}(14,17,20,23,26,29)\\
&=\frac{1}{\mu+1}\sum_{n=0}^\mu\binom{\mu+1}{n}B_n\cdot 14^{n-1}\biggl(
(17^{\mu+1-n}-3^{\mu+1-n})(-1)^{17}\\
&\quad +(20^{\mu+1-n}-6^{\mu+1-n})(-1)^{20}+(23^{\mu+1-n}-9^{\mu+1-n})(-1)^{23}\\
&\quad +(26^{\mu+1-n}-12^{\mu+1-n})(-1)^{26}+(29^{\mu+1-n}-1^{\mu+1-n})(-1)^{29}\\
&\quad +(46^{\mu+1-n}-4^{\mu+1-n})(-1)^{46}+(49^{\mu+1-n}-7^{\mu+1-n})(-1)^{49}\\
&\quad +(52^{\mu+1-n}-10^{\mu+1-n})(-1)^{52} +(55^{\mu+1-n}-13^{\mu+1-n})(-1)^{55}\\
&\quad +(58^{\mu+1-n}-2^{\mu+1-n})(-1)^{58}+(75^{\mu+1-n}-5^{\mu+1-n})(-1)^{75}\\
&\quad +(78^{\mu+1-n}-8^{\mu+1-n})(-1)^{78}+(81^{\mu+1-n}-11^{\mu+1-n})(-1)^{81}
\biggr)\,. 
\end{align*}
Then, we find that 
\begin{align}  
s_1^{(-1)}(14,17,20,23,26,29)&=-116\,,\notag\\
s_2^{(-1)}(14,17,20,23,26,29)&=-6380\,,\notag\\
s_3^{(-1)}(14,17,20,23,26,29)&=-375500\,,\notag\\
s_4^{(-1)}(14,17,20,23,26,29)&=-22771652\,,\notag\\
s_5^{(-1)}(14,17,20,23,26,29)&=-1406886596\,.
\label{eq:14-1}
\end{align}
In fact, the alternate sum is given by 
\begin{align*}  
&s_1^{(-1)}(14,17,20,23,26,29)\\
&=-1 + 2 - 3 + 4 - 5 + 6 - 7 + 8 - 9 + 10 - 11 + 12 - 13 - 15 + 16 + 18\\
&\quad - 19 - 21 + 22 + 24 - 25 - 27 + 30 
+ 32 - 33 - 35 + 36 + 38 - 39 - 41\\ 
&\quad + 44 - 47 + 50 - 53 - 61 + 64 - 67\\
&=-116\,.
\end{align*}

On the other hand, by applying Theorem \ref{th:wpw-arith-a1} as $a=14$, $d=3$, $k=6$, $q=2$ and $r=3$, for $\lambda=-1$, we can directly get the values in (\ref{eq:14-1}).

\section*{Acknowledgement}  

The author thanks the anonymous referee for all valuable comments and suggestions.

\end{document}